\documentclass[12pt, reqno]{amsart}

\usepackage{xcolor}  	
\usepackage{hyperref}
\hypersetup{
colorlinks,
   linkcolor={cyan!80!black},
   citecolor={cyan!80!black},
 urlcolor={cyan!80!black}
}

\usepackage{color}

\usepackage{subcaption}

\usepackage{graphicx}
\usepackage[section]{placeins}

\usepackage{amssymb}
\usepackage{amscd}
\usepackage{amsfonts}
\usepackage{float}
\usepackage{color}

\usepackage[
backend=biber,
style=alphabetic,
]{biblatex}

\renewbibmacro{in:}{}
\DeclareFieldFormat{title}{#1}

\DeclareFieldFormat[article]{title}{\mkbibemph{#1}}       
\DeclareFieldFormat[incollection]{title}{\mkbibemph{#1}}  
\DeclareFieldFormat[book]{title}{\mkbibemph{#1}}          
\DeclareFieldFormat[incollection]{booktitle}{#1}          
\DeclareFieldFormat[article]{journaltitle}{#1}            

\AtEveryBibitem{%
  \ifentrytype{misc}{\DeclareFieldFormat{title}{\mkbibemph{#1}}}{}}

\DeclareFieldFormat{eprint:eprint}{arXiv:\href{https://arxiv.org/abs/#1}{#1}}

\usepackage{bookmark}

\usepackage{amssymb}
\usepackage[top=1in, bottom=1in, left=1in, right=1in]{geometry}

\newtheorem{theorem}{Theorem}[section]
\newtheorem{lemma}{Lemma}[section]

\theoremstyle{definition}

\newtheorem{example}{Example}[section]

\newtheorem{proposition}{Proposition}[section]
\newtheorem{conjecture}{Conjecture}[section]

\newtheorem*{theorem*}{Theorem}

\theoremstyle{remark}
\newtheorem{remark}{Remark}[section]

\numberwithin{equation}{section}

\newcommand{\Mod}[1]{\ (\operatorname{mod} #1)}

\renewcommand{\Re}{\mathrm{Re}}

\renewcommand{\leq}{\leqslant}
\renewcommand{\geq}{\geqslant}

\begin{document}

\title{A central limit theorem for coefficients of $L$-functions in short intervals}


\author{}
\address{}
\curraddr{}
\email{}
\thanks{}

\author{Sun-Kai Leung}
\address{D\'epartement de math\'ematiques et de statistique\\
Universit\'e de Montr\'eal\\
CP 6128 succ. Centre-Ville\\
Montr\'eal, QC H3C 3J7\\
Canada}
\curraddr{}
\email{sun.kai.leung@umontreal.ca}
\thanks{}


\date{}

\dedicatory{}


\begin{abstract}
Assuming the generalized Lindel\"{o}f hypothesis (GLH), a weak version of the generalized Ramanujan conjecture and a Rankin--Selberg type partial sum estimate, we establish the normality of the sum of coefficients of a general $L$-function in short intervals of appropriate length. The novelty lies in the degree aspect under GLH. 
In particular, this generalizes the result of Hughes and Rudnick on lattice point counts in thin annuli.
\end{abstract}

\maketitle

\section{Introduction}

The distribution of arithmetic functions 
$a: \mathbb{N} \to \mathbb{C}$ 
is a central topic in analytic number theory. In particular, a natural question arises regarding the distribution of the partial sum
$\sum_{n \leq x}a(n)$ 
when $x \in [1, X]$ is chosen uniformly at random. One might expect that the behavior of this sum resembles that of a sum of independent and identically distributed random variables, suggesting it should follow a normal distribution according to the central limit theorem. Although it may come as a surprise that this is not typically the case, many arithmetic functions actually follow a Gaussian law in short intervals 
$[x,x+H]$ for some appropriate length $H=o(X)$
rather than in long intervals
$[1,x].$ 

If the length of short intervals $H \to \infty$ but $\log H = o(\log X)$ as $X \to \infty,$ i.e., very short, Davenport and Erd\H{o}s \cite{MR0055368}, and more recently Lamzouri \cite{MR3091515}, established central limit theorems for the short character sum $\sum_{x<n\leq x+H} \chi(n),$ where $\chi \Mod{q}$ is a Dirichlet character with $X=q.$ Similarly, assuming a uniform 
prime $k$-tuple conjecture, Montgomery and Soundararajan \cite{MR2104891} established a central limit theorem for the weighted prime count $\sum_{x<n \leq x+H} \Lambda(n)$ in almost the same range, except it is also required that $H/\log X \to \infty.$

On the other hand, if the length of short intervals $H=o(X)$ but $\log (X/H)  = o(\log X)$ as $X \to \infty,$ i.e., shortish, Harper \cite{harper2022note} proved very recently that for almost all primes $q,$ the short character sum $\sum_{x<n\leq x+H} \Bigl( \frac{n}{q} \Bigr)$ also follows a normal distribution with $X=q.$ Similarly, assuming the Riemann hypothesis, the linear independence hypothesis and its quantitative formulation, the author \cite[Theorem 9.1]{joint} has very recently proved that the weighted count of primes $\sum_{x <n \leq x+H} \Lambda(n)$ follows a normal distribution in the above range, with respect to the logarithmic density. In other words, as $u \in [1,U]$ varies uniformly, the weighted count of primes $\sum_{e^u <n \leq e^{u+\delta}} \Lambda(n)$ has a Gaussian limiting distribution, provided that $\delta=o(1)$ but $\log \frac{1}{\delta} = o(\log U)$ as $U \to \infty.$

The results of both Harper and the author rely heavily on the dual sum. In the case of short (primitive) character sums, the dual sum is P\'{o}lya's Fourier expansion 
\begin{align} \label{eq:gl1}
\sum_{x<n\leq x+H} \chi(n) =
\frac{2\epsilon_{\chi}}{\pi}
\sum_{0< n \leq \frac{q}{H}} \frac{\overline{\chi}(n)}{\sqrt{n}}
\cdot \frac{\sin \bigl(\pi n \frac{H}{q} \bigr)}{\sqrt{n/q}} \cdot
\cos \left( \frac{\pi n}{q} (2x+H)+\varphi \right)
+O(H\log q),
\end{align}
where $\varphi:=-\frac{\pi}{2} \cdot 1_{\chi \text{ odd}}$ (see \cite[Section 9.4]{MR2378655} for instance).
In the case of primes in short intervals, assuming the Riemann hypothesis, the dual sum is von Mangoldt's explicit formula 
\begin{align} \label{eq:glinfty}
\sum_{e^u <n \leq e^{u+\delta}} \Lambda(n) = 
(e^{u+\delta}-e^u)
-\frac{2}{\pi}e^{u/2}\sum_{0 < \widetilde{\gamma} \leq 1/\delta}
\frac{\sin(\pi \widetilde{\gamma} \delta)}{\widetilde{\gamma}}
\cdot \cos(\pi \widetilde{\gamma} (2u+\delta))+O(\delta e^u u^2),
\end{align}
where $\widetilde{\gamma}:=\gamma/2\pi,$ i.e., the normalized positive ordinate of a non-trivial Riemann zero (see \cite[Theorem 12.5]{MR2378655} for instance).

Here, primitive Dirichlet characters $\chi(n)$ are considered as ``$GL_1$-arithmetic functions" as the Dirichlet $L$-function $L(s,\chi)$ is of degree $1$ (see Section \ref{sec:pre} for definition). Meanwhile, let $\tau'_k(n):=|\{ (d_1, \ldots, d_k) \in \mathbb{N}^k_{>1} \,: \, n=d_1\cdots d_k \}|.$ Then, Linnik's identity (see \cite[Chapter 13.3]{MR2061214} for instance) states that 
\begin{align*}
\frac{\Lambda(n)}{\log n}=-\sum_{k=1}^{\infty} \frac{(-1)^k}{k}\tau'_k(n).
\end{align*}
Moreover, for each integer $k \geq 1,$ since $\tau_k(n)$ is a ``$GL_k$-arithmetic function" by the same reasoning, the von Mangoldt function $\Lambda(n)$ is regarded as a ``$GL_{\infty}$-arithmetic function". One may then ask: what about $GL_m$-arithmetic functions in general? Do they follow a normal distribution in short intervals? In this paper, we provide an affirmative answer conditionally, again by studying the dual sum, which is the Vorono\u{\i} summation formula.
\\~\\
\noindent\textit{Notation.} Throughout the paper, we use the standard big $O$ and little $o$ notations as well as the Vinogradov notations $\ll, \gg,$ where from now on, all implied constants depend on the fixed $L$-function $L(f,s)$ and subscripts are omitted.

\section{Definitions} \label{sec:pre}
We begin with introducing a general class of $L$-functions.
Throughout the paper, let 
\begin{align*}
L(f,s):=\sum_{n=1}^{\infty} \frac{\lambda_f(n)}{n^s}
\end{align*}
be an $L$-function of degree $m \geq 2$ which is self-dual, i.e., 
the coefficients $\lambda_f(n)$ are real, and is absolutely convergent for $\Re(s)>1.$
We denote the gamma factor as 
\begin{align*}
L_{\infty}(f,s):=\prod_{j=1}^m {\pi}^{-\frac{s+\kappa_j}{2}}
\Gamma\left( \frac{s+\kappa_j}{2} \right),
\end{align*}
where $\kappa_j$'s are the local parameters at $\infty.$ For convenience, we shall call 
\begin{align*}
k:= \sum_{j=1}^m \Re(\kappa_j)
\end{align*}
the weight of $L(f,s).$
The complete $L$-function is defined as 
\begin{align*}
\Lambda(f,s):=D^{\frac{s}{2}}L_{\infty}(f,s)L(f,s)
\end{align*}
with conductor $D,$ which admits an analytic continuation to $\mathbb{C}$ except possibly poles of finite order at $s=1.$ Moreover, it satisfies the function equation
\begin{align*}
\Lambda(f,s)=w \Lambda(f,1-s)
\end{align*}
with root number $w=\pm 1.$ 

\begin{remark}
The $L$-functions considered here do not necessarily have an Euler product, as the multiplicative structure of integers is not involved in establishing the normality.
\end{remark}

Let us denote $\Delta_f(x):=\sum_{n \leq x}\lambda_f(n)-\mathop{\mathrm{Res}}_{s=1} L(f,s)x^s/s.$ Then in this paper, 
we shall study the distribution of the remainder of the short coefficient sum
\begin{align*}
\Delta_f((x+\delta)^m)-\Delta_f(x^m)=\sum_{x^m<n \leq (x+\delta)^m} 
\lambda_f(n)-   \mathop{\mathrm{Res}}_{s=1} \left( L(f,s) \cdot
\frac{(x+\delta)^{ms}-x^{ms}}{s} \right) 
\end{align*}
conditioning on the following conjectures.

\begin{conjecture}[Generalized Lindel\"{o}f hypothesis (GLH)]
Let $t \in \mathbb{R}.$ Then for any $\epsilon>0,$ we have
$L(f,\frac{1}{2}+it) \ll_{\epsilon} D^{\epsilon} \prod_{j=1}^m (1+|t+\kappa_j|)^{\epsilon}.$
\end{conjecture}

\begin{conjecture}[Weak Generalized Ramanujan conjecture ($\text{GRC}^{\flat}$)]
There exists $\eta>0$ such that for any $n \geq 1,$ we have
$\lambda_f(n) \ll_{f} n^{\frac{1}{6}-\eta}.$
\end{conjecture}

Similar to Weyl's subconvexity bound, the exponent $1/6$ here is considered to be critical due to its applications to many classical problems in analytic number theory (see \cite[Section 5.3]{MR3020828} for instance), including shifted convolution sums and hyperbolic circle problem (see \cite[Section 8]{MR1923967} for instance).

\begin{remark}
It is well-known that the
normalized Fourier coefficients (Hecke eigenvalues) of holomorphic or Hecke--Maass cusp forms satisfy $ |\lambda_f(n)| \leq \tau(n)$ by Deligne \cite{MR0340258} or $ |\lambda_f(n)| \leq \tau(n) n^{7/64}$ by Kim and Sarnak \cite{MR1937203}, respectively. Therefore, the conjecture $\text{GRC}^{\flat}$ holds for such forms.
\end{remark}



\section{Main result} \label{sec:main}

With the definitions and conjectures in place, we shall state our main result as follows.

\begin{theorem} \label{thm:real}
Assume GLH and $\text{GRC}^{\flat}$ for a fixed $L$-function $L(f,s)$.
Let $X, \delta>0$ be real numbers for which $\delta=o_{X \to \infty}(1)$ but
$\log\left( 1/\delta \right)=o_{X \to \infty}(\log X).$ 
Suppose 
\begin{align} \label{eq:rankinselberg}
\sum_{n \leq y}\lambda_f(n)^2
=(c_f+o_{y \to \infty}(1)) y (\log  2y)^{r-1}
\end{align}
for some constant $c_f>0$ and integer $r \geq 1.$ Then,
if $x \in [X, 2X]$ is chosen uniformly at random, we have the convergence in distribution to a standard Gaussian
\begin{align} \label{eq:main1}
\frac{\Delta_f((x+\delta)^m)-\Delta_f(x^m)}{x^{\frac{m-1}{2}}\sigma_f(\delta)}
\xrightarrow[]{d} \mathcal{N}(0,1)
\end{align}
as $X \to \infty,$ where 
\begin{align*}
{\sigma_f(\delta)}^2:=
c_f m^r \delta \log^{r-1} \left(\frac{1}{\delta} \right),
\end{align*}
i.e., for any fixed real numbers $\alpha<\beta,$ we have
\begin{align*} 
\lim_{X \to \infty }\frac{1}{X}\mathop{\mathrm{meas}}\left\{ x \in [X,2X] \,: \, \frac{  \Delta_f((x+\delta)^m)-\Delta_f(x^m) }{x^{\frac{m-1}{2}}\sigma_f(\delta)} \in (\alpha,\beta] \right\}
=\frac{1}{\sqrt{2\pi}}\int_{\alpha}^{\beta}e^{-\frac{t^2}{2}}dt.
\end{align*}
For $m=2,$ the assumption of GLH can be dropped.

\end{theorem}

\begin{remark}
The Rankin--Selberg type partial sum estimate (\ref{eq:rankinselberg}) is known for many automorphic objects, especially the Hecke eigenvalues $A(n,1,\ldots,1)$ of Hecke--Maass cusp forms for $SL_m(\mathbb{Z})$ (see \cite[Theorem 6]{MR4206430}).
\end{remark}


The theorem can be regarded as the short-interval counterpart of the work by Kowalski and Ricotta \cite[Corollary B]{MR3248485} in the self-dual case, i.e., coefficients are always real. Compared to their work, the source of normality here stems from the oscillation of Bessel transforms instead of hyper-Kloosterman sums. In this paper, the main difficulty lies in controlling the errors incurred from approximating these Bessel transforms via the saddle point method (see \cite[p. 498]{MR2134400}). Therefore, unlike in their work, the Vorono\u{\i} summation formula with smooth weights offers little advantage. Nonetheless, our proof is more elementary since it does not rely on deep results regarding the equidistribution of hyper-Kloosterman sums.

It is certainly plausible to prove an analogous result for non-self-dual $L$-functions, i.e., coefficients are not always real. In such cases, the short coefficient sum would follow a bivariate Gaussian distribution in the complex plane. Besides, the proof should also work for $L$-functions belonging to the Selberg class (see \cite{MR1689554} for definition) conditionally. To avoid repetition, we refrain from pursuing these further here.

In the following, we provide several interesting examples.




\begin{example} Let $\lambda_f(n)$ be Hecke eigenvalues of a holomorphic Hecke cusp form or a Hecke--Maass cusp form $f$ of level $1$ for $n \in \mathbb{N}.$ Then Theorem \ref{thm:real} holds unconditionally with $m=2$
and ${\sigma_f(\delta)}^2=\frac{12}{\pi^2}L(1,\mathop{\mathrm{Sym}}^2 f)\delta.$ For $f$ holomorphic, this is essentially \cite[Theorem 1.4]{MR3504333} with respect to the ``square-root density" (see Remark \ref{rmk}).
\end{example}

\begin{example} \label{eq:tauk}
Given an integer $k \geq 2,$ let $\lambda_f(n)=\tau_k(n)$ for $n \in \mathbb{N}.$ Assume GLH for $k \geq 3$. Then Theorem \ref{thm:real} holds with $m=k$ and ${\sigma_f(\delta)}^2=c_fk^{(k^2)}\delta \log^{k^2-1}\left( \frac{1}{\delta} \right),$ where 
\begin{align*}
c_f=\prod_p \left( 1-\frac{1}{p} \right)^{k^2}\sum_{j=0}^{\infty}{j+k-1 \choose j}^2\frac{1}{p^j}.
\end{align*}
In particular, for $k=2,$ this is essentially  
\cite[Theorem 1.3]{MR3504333} with ${\sigma_f(\delta)}^2=\frac{16}{\pi^2}\delta \log^3\left( \frac{1}{\delta} \right)$ with respect to the ``square-root density".
\end{example}

\begin{example}
Given a Galois extension $K/\mathbb{Q}$ of degree $d,$ let $\lambda_f(n)=r_K(n)$ be the number of ideals of norm $n$ for $n \in \mathbb{N}.$  Assume GLH for $d \geq 3$. Then Theorem \ref{thm:real} holds with $m=d$
and ${\sigma_f(\delta)}^2=c_f d^d \delta 
\log^{d-1}\left( \frac{1}{\delta}\right).$\footnote{The satisfaction of assumption (\ref{eq:rankinselberg}) is guaranteed by \cite[Theorem 3]{MR0153643}.} 
In particular, for $K=\mathbb{Q}(i),$ this corresponds to the result on lattice point counts in thin annuli of Hughes and Rudnick \cite[Theorem 1.3]{MR2039790} with
${\sigma_f(\delta)}^2=\delta \log\left( \frac{1}{\delta} \right).$ 
\end{example}

\begin{remark} \label{rmk}
Suppose the conditions of Theorem \ref{thm:real} are satisfied. Let $H=H(x)=\delta x^{1-\frac{1}{m}}.$ Then it is certainly plausible to establish a central limit theorem for 
\begin{align*}
\frac{\Delta_{f}(x+H)-\Delta_{f}(x)}{x^{\frac{m-1}{2m}}\sigma_f(\delta)}
\end{align*}
as in \cite{MR3504333}. However, our proof suggests that the left-hand side of (\ref{eq:main1}) is equally natural (see \cite[Theorem 1.3]{MR2039790}). In other words, our normality holds with respect to the ``$m$-th root density" instead of the natural density, resembling the logarithmic density for primes in short intervals as discussed in the introduction.
\end{remark}

\section{Preliminaries}

The Vorono\u{\i} summation formula (without twists), as formulated by Friedlander and Iwaniec lies at the heart of the matter; although to prove the Theorem \ref{thm:real}, we do not apply the formula directly as their error term is not negligible for $m \geq 3.$

\begin{proposition}[Vorono\u{\i} summmation formula {\cite[Theorem 1.2]{MR2134400}}] \label{prop:fi}
Let $x \geq D^{1/2m}.$ Then for any $1 \leq N \leq x^{m},$ we have
\begin{align} \label{eq:Voronoi}
\Delta_f(x^m)= \frac{w}{\pi}\cdot x^{\frac{m-1}{2}}\sum_{n \leq N} 
\frac{\lambda_f(n)}{\sqrt{n}} \cdot \frac{\sin\left( 2\pi \breve{n}x+\varphi\right)}{\sqrt{\breve{n}}}+O_{\boldsymbol{\kappa}, \epsilon} ( 
(N/D)^{-\frac{1}{m}}x^{m-1+\epsilon}  ),
\end{align}
where $\breve{n}:=m(n/D)^{1/m}$ for $n \in \mathbb{N}$ and 
$\varphi:=\frac{\pi}{2}\left( \frac{m-1}{2}-k \right).$\footnote{A typographical error is present in the statement of \cite[Theorem 1.2]{MR2134400}: $(\pi m)^{-\frac{1}{2}}$ should be corrected to $\pi m^{-\frac{1}{2}}.$}
\end{proposition}




Instead of studying $x^{-(m-1)/2}(\Delta_f((x+\delta)^m)-\Delta_f(x^m))$ directly, we define a similar quantity
\begin{align*}
\Delta_f(x,\delta):=\frac{\Delta_f((x+\delta)^m)}{(x+\delta)^{(m-1)/2}}-\frac{\Delta_f(x^m)}{x^{(m-1)/2}},
\end{align*}
which is easier to handle, as the Vorono\u{\i} summation formula suggests that
\begin{align} \label{eq:main}
\Delta_f(x,\delta)
\approx \frac{2w}{\pi}\sum_{n \leq N} \frac{\lambda_f(n)}{\sqrt{n}} \cdot \frac{\sin(\pi \breve{n} \delta)}{\sqrt{\breve{n}}}\cdot \cos\left( \pi \breve{n}(2x+\delta)+\varphi\right)
\end{align}
whenever $m(N/D)^{1/m} \delta \to \infty$. This should be compared with (\ref{eq:gl1}) and (\ref{eq:glinfty}). The pseudo-independence of $\cos\left( \pi \breve{n}(2x+\delta)+\varphi\right)$ for $x \in [X, 2X]$ is guaranteed by Lemma \ref{lem:li}, so that heuristically, we have
\begin{align*}
\Delta_f(x,\delta)
\approx \frac{2w}{\pi}\sum_{n \leq N} \frac{\lambda_f(n)}{\sqrt{n}} \cdot \frac{\sin(\pi \breve{n} \delta)}{\sqrt{\breve{n}}} \cdot \mathbb{X}_n,
\end{align*}
where $\mathbb{X}_n \sim \Re(\mathbb{U}(S^1))$ are independent and identically distributed random variables. Then, it follows from the classical central limit theorem that 
$\Delta_f(x,\delta)$ is approximately normal with mean $0$ and variance 
\begin{align} \label{eq:varimp}
\frac{2}{\pi^2} \sum_{n = 1}^{\infty}
\frac{\lambda_f(n)^2}{n} \cdot \frac{\sin^2(\pi \breve{n} \delta)}{\breve{n}}.
\end{align}
To argue rigorously, we compute the moments of (\ref{eq:main}) as in \cite{MR2039790} and  \cite{MR3504333}.

\begin{lemma} \label{lem:li}
Let $q_1, \ldots, q_k$ be distinct $m$-th power-free positive integers. Then $q_1^{1/m}, \ldots, q_k^{1/m}$ are linearly independent over $\mathbb{Q}.$
\end{lemma}

\begin{proof}
Unlike in \cite{MR2039790} and \cite{MR3504333}, the applicability of \cite[Theorem 2]{MR2327} is not straightforward. Instead, we provide a different proof. Suppose $\sum_{j=1}^k c_j q_j^{1/m}=0$ for some $c_j \in \mathbb{Q}.$ Let $j_0=1,\ldots,k$ be fixed. Then
\begin{align*}
c_{j_0}=-\sum_{\substack{j=1\\j \neq j_0}}^k c_j \left( q_j/q_{j_0}\right)^{1/m}.
\end{align*}
Let $r_j=(q_j/q_{j_0})^{1/m}$ and $K$ be the finite extension of $\mathbb{Q}$ by adjoining $r_j$ for $j \neq j_0.$ Taking the trace $\operatorname{Tr}_{\mathbb{Q}}^{K}$ on both sides gives 
\begin{align} \label{eq:trace}
[K:\mathbb{Q}]c_{j_0}=-\sum_{\substack{j=1\\j \neq j_0}}^k c_j \operatorname{Tr}_{\mathbb{Q}}^{K} (r_j).
\end{align}
Note that by assumption, there is the smallest $1<m_j \leq m$ for which $r_j^{m_j} \in \mathbb{Q}$ and the minimal polynomial of $r_j$ is $x^{m_j}-r_j^{m_j}.$ Therefore, we have
\begin{align*}
\operatorname{Tr}_{\mathbb{Q}}^{K} 
( r_j )
=[K:\mathbb{Q}(r_j)]\sum_{i=1}^{m_j} \zeta_{m_j}^i r_j =0,
\end{align*}
so that the right-hand side of (\ref{eq:trace}) vanishes and $c_{j_0}=0.$ Since $j_0$ is chosen arbitrarily, we conclude that $c_j=0$ for any $j=1,\ldots,k$ and the lemma follows.
\end{proof}

As we shall prove Theorem \ref{thm:real} via the method of moments, a quantitative lower bound for non-vanishing alternating sums of $q_j^{1/m}$ is required.

\begin{lemma} \label{lem:diss}
For $j=1,\ldots,k,$ let $1 \leq n_j \leq N$ and $\epsilon_j \in \{\pm1\}.$
Suppose
\begin{align*}
\sum_{j=1}^k \epsilon_j n_j^{1/m} \neq 0.
\end{align*}
Then 
\begin{align*}
\left|\sum_{j=1}^k \epsilon_j n_j^{1/m}  \right| \geq (kN^{1/m})^{-(m^{k}-1)}.
\end{align*}

\begin{proof}
Let $G$ be the Galois group of the field extension $K/\mathbb{Q},$ where $K$ is the Galois closure of $\mathbb{Q}( n_1^{1/m}, \ldots, n_k^{1/m}  )$. Note that $|G| = [K:\mathbb{Q}] \leq m^k.$ Then, we consider the product
\begin{align*}
P:=\prod_{\sigma \in G} \sigma \left( \sum_{j=1}^k \epsilon_j n_j^{1/m}  \right)=\prod_{\sigma \in G} \left( \sum_{j=1}^k \epsilon_j \sigma (n_j^{1/m} ) \right).
\end{align*}
Since the set of algebraic integers forms a 
ring, the product $P$ is an algebraic integer. Moreover, since 
$P$ is fixed by $G$, it is also a rational number by Galois theory, and hence $P$ is an integer. As the maps $\sigma \in G$ are injective and $\sum_{j=1}^k \epsilon_j n_j^{1/m} \neq 0$ by assumption, no terms in $P$ vanishes, and thus $|P| \geq 1.$ Then, we have
\begin{align*}
\left|\sum_{j=1}^k \epsilon_j n_j^{1/m}  \right|
\geq \prod_{\substack{\sigma \in G\\ \sigma \neq e}} \left| \sum_{j=1}^k \epsilon_j
\sigma (n_j^{1/m} )
\right|^{-1} \geq (kN^{1/m})^{-(m^k-1)},
\end{align*}
and the lemma follows.
\end{proof}

\end{lemma}

Finally, we provide an estimate for the variance (\ref{eq:varimp}).

\begin{lemma} \label{lem:selbergrankin}
Suppose the estimate (\ref{eq:rankinselberg}) holds.
Then as $\delta \to 0^+,$ we have
\begin{align*}
\frac{2}{\pi^2} \sum_{n = 1}^{\infty}
\frac{\lambda_f(n)^2}{n} \cdot \frac{\sin^2(\pi \breve{n} \delta)}{\breve{n}}=&
(c_f+o(1)) m^r \delta \log^{r-1} \left(\frac{1}{\delta} \right)\\
=&(1+o(1))\sigma_f(\delta)^2.
\end{align*}
\end{lemma}

\begin{proof}
Let 
$
S(y):=\sum_{n \leq y}\lambda_f(n)^2
.$
Then by partial summation, we have
\begin{align*}
\frac{2}{\pi^2} \sum_{n = 1}^{\infty}
\frac{\lambda_f(n)^2}{n} \cdot \frac{\sin^2(\pi \breve{n} \delta)}{\breve{n}}
=& \frac{2}{\pi^2} \int_1^{\infty}
\frac{\sin^2(\pi m (y/D)^{1/m}\delta)}{m(y/D)^{1/m}} 
\cdot \frac{dS(y)}{y}.
\end{align*}
We write $S(y)=M(y)+E(y),$ where the main term $M(y):=c_f y (\log y)^{r-1}$ and the error term $E(y)=o(y(\log 2y)^{r-1}).$ Then this becomes
\begin{align} \label{eq:decomp}
\frac{2}{\pi^2}&
\int_{1}^{\infty}
\frac{\sin^2(\pi m (y/D)^{1/m} \delta)}{m (y/D)^{1/m}}
\cdot \frac{d(c_fy(\log y)^{r-1})}{y}
+ 
\frac{2}{\pi^2}
\int_{1}^{\infty}
\frac{\sin^2(\pi m (y/D)^{1/m} \delta)}{m (y/D)^{1/m}}
\cdot \frac{dE(y)}{y} \nonumber \\
= \frac{2}{\pi^2} & \cdot c_f
\int_{1}^{\infty}
\frac{\sin^2(\pi m (y/D)^{1/m} \delta)}{m (y/D)^{1/m} y}
\cdot (\log y)^{r-1} dy \nonumber \\
+ 
\frac{2}{\pi^2} & \cdot c_f
\int_{1}^{\infty}
\frac{\sin^2(\pi m (y/D)^{1/m} \delta)}{m (y/D)^{1/m} y}
\cdot
(r-2)(\log y)^{r-2} dy \nonumber\\
+ \frac{2}{\pi^2} & \int_{1}^{\infty}
E(y)
d \left(\frac{\sin^2(\pi m (y/D)^{1/m} \delta)}{m (y/D)^{1/m} y} \right)+O(\delta^2) 
=:I_1+I_2+I_3+O(\delta^2).
\end{align}
Let us estimate the integral $I_1$ first.
Making the change of variables $u=m(y/D)^{1/m}\delta,$ we have
\begin{align*}
I_1=&
\frac{2}{\pi^2} \cdot c_f 
\int_{m\delta/D^{1/m}}^{\infty}
\sin^2 (\pi u) \left( D \left( \frac{u}{m\delta} \right)^m \right)^{-(1+\frac{1}{m})}
\log^{r-1} \left( D \left( \frac{u}{m\delta} \right)^m \right)
\frac{D^{1+\frac{1}{m}}}{m\delta} 
\left( \frac{u}{m\delta} \right)^{m-1} du \\
=&
2c_f m \delta \int_{m\delta/D^{1/m}}^{\infty}
\left( \frac{\sin(\pi u)}{\pi u} \right)^2
\log^{r-1} \left( D \left( \frac{u}{m\delta} \right)^m \right) du \\
=& 2c_f m \delta 
\left\{\int_{0}^{\infty} -\int_{0}^{m\delta/D^{1/m}} \right\}
\left( \frac{\sin(\pi u)}{\pi u} \right)^2
 \left( m\log \frac{1}{\delta}+\log \left(D \left( \frac{u}{m} \right)^m \right) \right)^{r-1} du \\
=& c_f m^r \delta \log^{r-1} \left(\frac{1}{\delta} \right)
+O \left( \delta \log^{r-2}  \left(\frac{1}{\delta} \right) \right).
\end{align*}
Arguing similarly, the integral $I_2$ is $\ll \delta \log^{r-2}  \left(\frac{1}{\delta} \right).$ It remains to estimate the integral $I_3.$ Using the product rule, we have
\begin{gather*}
I_3=\int_{1}^{\infty}
E(y) 
\left( \frac{\pi m \delta}{D^{1/m}} \sin 
\left(2 \pi m \left( \frac{y}{D} \right)^{\frac{1}{m}} \delta  \right) y^{\frac{2}{m}} - \frac{m+1}{D^{1/m}} 
\sin^2 \left( \pi m \left( \frac{y}{D} \right)^{\frac{1}{m}} \delta \right)  y^{\frac{1}{m}} \right) \\
\cdot  
\left( m  \left( \frac{y}{D} \right)^{\frac{1}{m}} y \right)^{-2} dy,
\end{gather*}
which is 
\begin{align*}
o \left( \delta
\int_{1}^{\infty}
\sin 
\left(2 \pi m \left( \frac{y}{D} \right)^{\frac{1}{m}} \delta  \right)\frac{(\log 2y)^{r-1}}{y} dy
+
\int_{1}^{\infty}
\sin^2
\left( \pi m \left( \frac{y}{D} \right)^{\frac{1}{m}} \delta  \right)
\frac{(\log 2y)^{r-1}}{y^{1+\frac{1}{m}}} dy
\right).
\end{align*}
Decomposing the range of integration $[1, \infty)$ into $[1,\delta^{-m}] \cup [\delta^{-m}, \infty),$ one can show that both integrals here are $\ll \delta \log^{r-1} \left( \frac{1}{\delta} \right),$ so that $I_3 = o \left(\delta \log^{r-1} \left( \frac{1}{\delta} \right) \right).$ Combining with the estimates for integrals $I_1$ and $I_2$, the lemma follows from (\ref{eq:decomp}).
\end{proof}

\section{Mean and variance}


As in \cite{MR2039790} and \cite{MR3504333}, instead of the Lebesgue measure, we first integrate with respect to a nonnegative Schwartz function $W$ of unit mass supported on $(1/2,5/2).$ In particular, its Fourier transform $\widehat{W}$ decays rapidly in the sense that for any $A>0,$ we have
\begin{align*}
 \widehat{W}(\xi) \ll_A \frac{1}{(1+|\xi|)^A},
\end{align*}
where the Fourier transform $\widehat{W}$ is defined as
\begin{align*}
\widehat{W}(\xi):=\int_{-\infty}^{\infty}W(x)e^{-2\pi ix\xi}dx.
\end{align*}
For convenience, we write
\begin{align*}
\mathbb{E}_{x \sim X}^W\left( f(x) \right):= \frac{1}{X}\int_{-\infty}^{\infty} f(x)W\left( \frac{x}{X} \right)
dx.
\end{align*}

In view of (\ref{eq:main}), we shall establish a central limit theorem for the finite approximation
\begin{align*}
\Delta_f(x,\delta;N) :=
\frac{2w}{\pi}\sum_{n \leq N} \frac{\lambda_f(n)}{\sqrt{n}} \cdot \frac{\sin(\pi \breve{n} \delta)}{\sqrt{\breve{n}}}\cdot \cos\left( \pi \breve{n}(2x+\delta)+\varphi\right).
\end{align*}
We also denote the truncated variance as
\begin{align*}
\sigma_{f}(\delta;N)^2:=\frac{2}{\pi^2}
\sum_{n \leq N} 
\frac{\lambda_f(n)^2}{n} \cdot \frac{\sin^2(\pi \breve{n} \delta)}{\breve{n}}.
\end{align*}

Let us begin with computing the mean and variance (or second moment) of $\Delta_f(x,\delta;N).$ 

\begin{lemma} \label{lem:mean}
Let $X, \delta>0, N \geq 2.$ 
Then
\begin{align*}
\mathbb{E}_{x \sim X}^W\left( \Delta_f(x,\delta;N) \right)
\ll_A \delta X^{-A}.
\end{align*}
In particular, if $\delta=o(1)$ as $X \to \infty,$ then $\mathbb{E}_{x \sim X}^W\left( \Delta_f(x,\delta;N) \right)=o( \sigma_{f}(\delta;N) ).$ 

\begin{proof}
Let $1 \leq n \leq N.$ Then 
\begin{align*}
\frac{1}{X}\int_{-\infty}^{\infty} \cos\left( \pi \breve{n}(2x+\delta)+\varphi\right) W\left( \frac{x}{X}\right)dx
&=\Re \left( \widehat{W}\left( -\breve{n}X \right) e^{i(\pi \breve{n} \delta +\varphi)} \right)\\
&\ll_A (\breve{n}X)^{-A}.
\end{align*}
Recall that $\breve{n}:=m(n/D)^{1/m}.$ Therefore, we have
\begin{align*}
\mathbb{E}_{x \sim X}^W\left( \Delta_f(x,\delta;N) \right)
&=\frac{2w}{\pi}\sum_{n \leq N} \frac{\lambda_f(n)}{\sqrt{n}} \cdot \frac{\sin(\pi \breve{n} \delta)}{\sqrt{\breve{n}}} \cdot
O_A \left( (\breve{n}X)^{-A} \right) \\
&\ll_A \sum_{n \leq N} \frac{|\lambda_f(n)|}{\sqrt{n}}
\cdot \delta \sqrt{\breve{n}} \cdot (\breve{n}X)^{-A} \\
&\ll_A \delta X^{-A}.
\end{align*}
Moreover, this is $\ll_A \sigma_{f}(\delta;N) X^{-A}$ since $\sigma_{f}(\delta;N)^2 \gg \delta^2 $ when considering only the first term with $n=1$. Therefore, the lemma follows.
\end{proof}

\end{lemma}

\begin{lemma} \label{lem:var}

Let $X, \delta>0, N \geq 2.$ Suppose $N \leq X^{\frac{m(1-\theta)}{m-1}}$ for some fixed small $\theta>0.$ Then
\begin{align*}
\mathbb{E}_{x \sim X}^W( \Delta_f(x,\delta;N)^2 )
=\sigma_{f}(\delta;N)^2
+O_A \left( \delta^2 X^{-A} \right).
\end{align*}
In particular, if $\delta=o(1)$ as $X \to \infty,$ then $\mathbb{E}_{x \sim X}^W( \Delta_f(x,\delta;N)^2 )=(1+o(1))\sigma_f(\delta;N)^2.$

\begin{proof}
Let $1 \leq n_1, n_2 \leq N.$ Then
\begin{gather} 
\frac{1}{X} \int_{-\infty}^{\infty}
\cos\left( \pi \breve{n_1}(2x+\delta)+\varphi\right) 
\cos\left( \pi \breve{n_2}(2x+\delta)+\varphi\right)
W\left( \frac{x}{X}\right)dx \nonumber \\
= \frac{1}{2} \cdot \Re \left( \widehat{W}((\breve{n_1}+\breve{n_2})X)e^{-i(\pi(\breve{n_1}+\breve{n_2})\delta+2\varphi)}+ 
\widehat{W}((\breve{n_1}-\breve{n_2})X)
e^{-i\pi(\breve{n_1}-\breve{n_2})\delta}
\right) \nonumber \\
=
\begin{cases} \label{eq:coscos}
1/2+O_A \left( ((\breve{n_1}+\breve{n_2})X)^{-A} \right) & \mbox{{\normalfont if $n_1=n_2,$ } } \\
\hfil O_A \left( ((\breve{n_1}-\breve{n_2})X)^{-A} \right) & \mbox{{\normalfont if $n_1 \neq n_2.$ } }
\end{cases}
\end{gather}
Since by definition 
\begin{align*}
\breve{n_1}-\breve{n_2} = \frac{m}{D^{1/m}}(n_1^{1/m}-n_2^{1/m}) \gg N^{\frac{1}{m}-1} 
\end{align*}
if $n_1 \neq n_2,$ the expression (\ref{eq:coscos}) is
\begin{align*}
\begin{cases}
1/2+O_A \left( ((\breve{n_1}+\breve{n_2})X)^{-A} \right) & \mbox{{\normalfont if $n_1=n_2,$ } }\\
\hfil O_A ( (  N^{\frac{1}{m}-1} X )^{-A} ) & \mbox{{\normalfont if $n_1 \neq n_2.$ } }
\end{cases}
\end{align*}
Therefore, we have
\begin{gather*}
\mathbb{E}_{x \sim X}^W( \Delta_f(x,\delta;N)^2)=
\frac{4}{\pi^2}  \sum_{n \leq N} 
\frac{\lambda_f(n)^2}{n} \cdot \frac{\sin^2(\pi \breve{n} \delta)}{\breve{n}} \cdot \left( \frac{1}{2}+O_A \left( (2\breve{n}X)^{-A} \right) \right) \\
+\sum_{n_1 \neq n_2 \leq N}
\frac{\lambda_f(n_1)\lambda_f(n_2)}{\sqrt{n_1n_2}} \cdot \frac{\sin(\pi \breve{n_1} \delta)\sin(\pi \breve{n_2} \delta)}{\sqrt{\breve{n_1}\breve{n_2}}} \cdot
O_A ( (  XN^{\frac{1}{m}-1}  )^{-A} ) \\
=\frac{2}{\pi^2}  \sum_{ n \leq N} 
\frac{\lambda_f(n)^2}{n} \cdot \frac{\sin^2(\pi \breve{n} \delta)}{\breve{n}}+
O_A ( \delta^2 X^{-A} ) +O_A ( \delta^2 N^2  (   X  N^{\frac{1}{m}-1} )^{-A}   ).
\end{gather*}
Note that the first error term is absorbed by the second. Also, since by assumption $N \leq X^{\frac{m(1-\theta)}{m-1}},$ this is
\begin{align*}
\sigma_f(\delta;N)^2
+O_B \left( \delta^2 X^{-B} \right).
\end{align*}
Moreover, this is $( 1+O_B (X^{-B}) )
\sigma_{f}(\delta;N) ^2 $ since $\sigma_{f}(\delta;N)^2 \gg \delta^2.$ Therefore, the lemma follows.
\end{proof}

\end{lemma}

In fact, the truncated variance $\sigma_{f}(\delta;N)^2$ is a good approximant.

\begin{lemma} \label{lem:varianceN}
Let $\delta>0, N \geq 2.$ Then 
\begin{align*}
\sigma_{f}(\delta;N)^2 =  \sigma_{f}(\delta)^2+
O ( N^{-\frac{1}{m}}\log^{r-1} N ).
\end{align*}
In particular, if $\delta^{m}N \to \infty$ as $\delta \to 0^+,$ then 
\begin{align*}
\sigma_{f}(\delta;N)^2=&(1+o(1)) \sigma_{f}(\delta)^2 \\
=&(c_f+o(1)) m^r \delta \log^{r-1} \left(\frac{1}{\delta} \right).
\end{align*}

\begin{proof}

By definition, we have 
\begin{align*}
\sigma_{f}(\delta)^2 = \sigma_{f}(\delta;N)^2+
\frac{2}{\pi^2}\sum_{n > N} 
\frac{\lambda_f(n)^2}{n} \cdot \frac{\sin^2(\pi \breve{n} \delta)}{\breve{n}}.
\end{align*}
Assuming (\ref{eq:rankinselberg}), the last sum is
\begin{align*}
\ll \sum_{n>N} \frac{\lambda_f(n)^2}{n} \cdot \frac{1}{n^{1/m}}
\ll \frac{\log^{r-1} N}{N^{1/m}}.
\end{align*}
 Then, the lemma follows from Lemma \ref{lem:selbergrankin}.
\end{proof}

\end{lemma}

\section{Higher moments}

Given $k \geq 3, \boldsymbol{\epsilon} \in \{ \pm 1\}^k, \boldsymbol{r} \in \mathbb{N}^k$ and a subset $S \subseteq [k]:=\{ 1,2 ,\ldots, k\},$ we define
\begin{align*}
\langle \boldsymbol{\epsilon}, \boldsymbol{r}\rangle|_{S}:=\sum_{j \in S} \epsilon_j r_j.
\end{align*}
Also, given a $m$-th power-free number $q \geq 1,$ let us denote
\begin{align*}
D_{q}(S):=&\left( \frac{w}{\pi} \right)^{|S|} 
\underset{\substack{\epsilon_j, r_j , \, j \in S \\ \langle \boldsymbol{\epsilon}, \boldsymbol{r} \rangle |_{S}=0}}{\sum \sum}
\prod_{j \in S} \frac{ \lambda_f(n_j)}{\sqrt{n_j }} \cdot
\frac{ \sin \left( \pi \breve{n_j}\delta\right)}{\sqrt{\breve{n_j}}} \cdot e^{i\epsilon_j \varphi},
\end{align*}
where $\epsilon_j =  \pm 1, 1 \leq r_j \leq (N/q)^{1/m}$ for each $j \in S,$ and $n_j=n_j(q,r_j):=q r_j^m.$

To compute higher moments, we require an estimate for the $L^1$-norm of $D_q(S)$ over $m$-th power-free numbers.

\begin{lemma} \label{lem:sumdqs}
Assume $\text{GRC}^{\flat}.$ Let $\delta>0, N \geq 2.$ Suppose $\delta^{m}N \to \infty$ as $\delta \to 0^+.$ Then

\begin{align*}
\frac{1}{\sigma_f(\delta; N)^{|S|} }\sum_{\substack{q \leq N\\ \text{\normalfont $m$-th power-free}}} \left|D_q(S) \right|
=
\begin{cases}
\hfil 0 & \mbox{{\normalfont if $|S|=1,$ }} \\
\hfil 1 & \mbox{{\normalfont if $|S|=2,$ }} \\
O  (\delta^{m ((\frac{1}{3}+\eta)|S|-1)} ) & \mbox{{\normalfont if $|S| \geq 3$. }} 
\end{cases}
\end{align*}

\begin{proof}
We adapt the proof of \cite[Lemma 9]{MR2039790} with $\text{GRC}^{\flat}.$
By definition, for $|S|=1,$ the sum $D_q(S)$ is empty. Also, if $|S|=2,$ then the sum is simply $\sigma_f(\delta;N)^2.$ Therefore, it remains to deal with the case $|S| \geq 3.$ Using the assumption $\text{GRC}^{\flat},$ we have
\begin{align} \label{eq:dqs}
\sum_{\substack{q \leq N\\ \text{\normalfont $m$-th power-free}}} \left|D_q(S) \right| \ll_k \sum_{q=1}^{\infty} 
\frac{E_q(S)}{q^{(\frac{1}{3}+\frac{1}{2m}+\eta) |S|}},
\end{align}
where
\begin{align} \label{eq:eqs}
E_q(S) :=  \underset{\substack{\epsilon_j, r_j , \, j \in S \\ \langle \boldsymbol{\epsilon}, \boldsymbol{r} \rangle |_{S}=0}}{\sum \sum} 
\prod_{j \in S} 
\frac{|\sin \left( \pi \breve{n_j}\delta\right)|}{r_j^{\frac{m}{3}+\frac{1}{2}+m\eta  } }.
\end{align}

Note that $E_q(S) \ll_{k} 1$ for any $q.$ Moreover, we show that for $q \leq \delta^{-m},$ a more careful treatment of the product yields a sharper upper bound. In order for $ \langle \boldsymbol{\epsilon}, \boldsymbol{r} \rangle |_{S}=0$ to hold, at least two of the $\epsilon$'s must have opposite signs. For convenience, we assume without loss of generality that $S=\{ 1,2,\ldots, |S| \}$ and $\epsilon_{|S|}=-1, \epsilon_{|S|-1}=+1$ so that
\begin{align*}
r_{|S|}=r_{|S|-1}+\sum_{j=1}^{|S|-2} \epsilon_j r_j.
\end{align*}
In order for $r_{|S|}, r_{|S|-1}  \geq 1$ to hold, we have $r_{|S|-1} \geq 1+\max \{ 0, -\sum_{j=1}^{|S|-2} \epsilon_j r_j \}.$
Therefore, the expression (\ref{eq:eqs}) is
\begin{align*}
E_q(S) \leq  2 
& \sum_{\substack{\epsilon_j=\pm 1\\j=1,\ldots,|S|-2}} 
\sum_{\substack{r_j \geq 1\\j=1,\ldots,|S|-2}}
\sum_{r_{|S|-1} \geq 1+\max \{ 0, -\sum_{j=1}^{|S|-2} \epsilon_j r_j \}} \\
& \left(\prod_{j=1}^{|S|-1} \frac{|\sin \left( \pi 
mr_j(q/D)^{1/m} \delta\right)|}{r_j^{ \frac{m}{3}+\frac{1}{2}+m\eta }} \right) \frac{|\sin \left( \pi 
m\left( r_{|S|-1}+\sum_{j=1}^{|S|-2} \epsilon_j r_j \right)(q/D)^{1/m} \delta\right)| }{\left( r_{|S|-1}+\sum_{j=1}^{|S|-2} \epsilon_j r_j \right)^{ \frac{m}{3}+\frac{1}{2}+m\eta }},
\end{align*}
which is
\begin{align} \label{eq:crazyintegral}
\ll \sum_{\substack{\epsilon_j=\pm 1\\j=1,\ldots,|S|-2}} 
&\underset{\substack{x_1,\ldots,x_{|S|-2}\geq 1\\x_{|S|-1}\geq 1+\max \{ 0, -\sum_{j=1}^{|S|-2} \epsilon_j x_j \}}}{\int \cdots \int}\left(\prod_{j=1}^{|S|-1} \frac{|\sin \left( \pi 
mx_j(q/D)^{1/m} \delta\right)|}{x_j^{ \frac{m}{3}+\frac{1}{2}+m\eta }} \right) \nonumber\\
& \cdot \frac{|\sin \left( \pi 
m\left( x_{|S|-1}+\sum_{j=1}^{|S|-2} \epsilon_j x_j \right)(q/D)^{1/m} \delta\right)| }{\left( x_{|S|-1}+\sum_{j=1}^{|S|-2} \epsilon_j x_j \right)^{ \frac{m}{3}+\frac{1}{2}+m\eta }}
\, dx_1 \cdots dx_{|S|-1}.
\end{align}
If $\frac{m}{3}+\frac{1}{2}+m\eta>2,$ then by using $|\sin x| \leq |x|$ for any $x \in \mathbb{R},$ this is $\ll_k (q^{1/m}\delta)^{|S|}.$ Otherwise, if $\frac{m}{3}+\frac{1}{2}+m\eta < 2$ (replace $\eta$ by $\eta/2$ if equality holds), then by
making the change of variables $y_j=mx_j(q/D)^{1/m} \delta$ for $j=1,\ldots,|S|-1,$ the expression (\ref{eq:crazyintegral}) becomes
\begin{gather*}
\left( m\left( \frac{q}{D} \right)^{1/m} \delta \right)^{1+\left( 
\frac{m}{3}-\frac{1}{2}+m\eta  \right)|S|}\sum_{\substack{\epsilon_j=\pm 1\\j=1,\ldots,|S|-2}} 
\underset{\substack{y_1,\ldots,y_{|S|-2}\geq m(q/D)^{1/m}\delta\\y_{|S|-1}\geq m(q/D)^{1/m}\delta+\max \{ 0, -\sum_{j=1}^{|S|-2} \epsilon_j y_j \}}}{\int \cdots \int} \\
\left(\prod_{j=1}^{|S|-1}\frac{|\sin \left( \pi 
y_j \right)|}{y_j^{ \frac{m}{3}+\frac{1}{2}+m\eta }} \right) 
 \frac{\left|\sin \left( \pi 
\left( y_{|S|-1}+\sum_{j=1}^{|S|-2} \epsilon_j y_j \right)\right)\right| }{\left( y_{|S|-1}+\sum_{j=1}^{|S|-2} \epsilon_j y_j \right)^{ \frac{m}{3}+\frac{1}{2}+m\eta }}
\, dy_1 \cdots dy_{|S|-1}.
\end{gather*}
Note that the multiple integral here is uniformly bounded. Therefore, we conclude that
\begin{align} \label{eq:eqsbound}
E_q(S) \ll_{k} 
\begin{cases}
\hfil (q^{1/m}\delta)^{|S|} & \mbox{{\normalfont if $q \leq \delta^{-m}$ and $\frac{m}{3}+\frac{1}{2}+m\eta>2,$ }} 
\\
( q^{1/m} \delta )^{1+\left( 
\frac{m}{3}-\frac{1}{2}+m\eta  \right)|S|} & \mbox{{\normalfont if $q \leq \delta^{-m}$ and $\frac{m}{3}+\frac{1}{2}+m\eta<2,$ }} \\
  \hfil 1 & \mbox{{\normalfont if $q > \delta^{-m}.$ }}
\end{cases}
\end{align}
To avoid repetitions, we only consider the harder case $\frac{m}{3}+\frac{1}{2}+m\eta<2$ here. The following argument also applies to the remaining case. Substituting this into (\ref{eq:dqs}) gives
\begin{align} \label{eq:dqs2}
\sum_{\substack{q \leq N\\ \text{\normalfont $m$-th power-free}}} \left| D_q(S) \right|\ll_{k} 
 &\sum_{q \leq  \delta^{-m} } 
 q^{-\left(\frac{1}{3}+\frac{1}{2m}+\eta \right) |S|}
 \left( q^{1/m} \delta \right)^{1+\left( 
\frac{m}{3}-\frac{1}{2}+m\eta  \right)|S|} \nonumber\\
+   &\sum_{q > \delta ^{-m}} 
 q^{-\left(\frac{1}{3}+\frac{1}{2m}+\eta \right) |S|}.
\end{align}
Appealing to Lemma \ref{lem:varianceN}, the second sum here is
\begin{align} \label{eq:2ndsum}
\ll
\delta^{-m+\left(\frac{m}{3}+\frac{1}{2}+m\eta \right)|S|}
\ll  \delta^{m \left(\left(\frac{1}{3}+\eta \right)|S|-1\right)} \sigma_f(\delta;N)^{|S|}.
\end{align}
On the other hand, since 
\begin{gather*}
 \sum_{q \leq  \delta^{-m}} 
 q^{-\left( \frac{1}{3}+\frac{1}{2m}+\eta \right)|S|}
 \cdot 
 q^{\frac{1}{m}+\left( \frac{1}{3}-\frac{1}{2m}+\eta \right)|S|} 
 =\sum_{q \leq  \delta^{-m}} q^{-\frac{1}{m}(|S|-1)} \\
 \ll_k
 \begin{cases}
\hfil \delta ^{|S|-1-m} & \mbox{{\normalfont if $|S|\leq m,$ }} \\
\hfil \log \frac{1}{\delta} & \mbox{{\normalfont if $|S|= m+1,$ }} \\
\hfil 1  &  \mbox{{\normalfont if $|S| >  m+1,$ }} 
 \end{cases}
\end{gather*}
the first sum of (\ref{eq:dqs2}) is
\begin{align*}
\ll_{k}
\begin{cases}
\hfil \delta^{-m+\left( \frac{m}{3}+\frac{1}{2}+m\eta \right)|S|}  & \mbox{{\normalfont if $|S| \leq m,$ }} \\
\hfil \delta^{1+\left( \frac{m}{3}+\frac{1}{2}+m\eta \right)|S|} \log \frac{1}{\delta} 
& \mbox{{\normalfont if $|S| =m+1,$ }} \\
\hfil \delta^{1+\left( \frac{m}{3}+\frac{1}{2}+m\eta \right )|S|} & \mbox{{\normalfont if $|S| >m+1.$ }}
\end{cases}
\end{align*}
Combining with (\ref{eq:2ndsum}), the lemma follows.
\end{proof}

\end{lemma}

With the lemma in hand, we can compute higher moments.

\begin{proposition} \label{prop:hm}
Given $X, \delta>0$ and $ k\geq 3,$ let $2 \leq N \leq
X^{\frac{m(1-\theta)}{m^{k}-1}}$ for some fixed small $\theta>0.$ Suppose $\delta=o(1)$ but $\delta^{m}N \to \infty$ as $X \to \infty.$ 
Then
\begin{align*}
\mathbb{E}_{x \sim X}^W( \Delta_f(x,\delta;N)^k )=(\mu_k+o_k(1))\sigma_f(\delta;N)^k,
\end{align*}
where
\begin{align*}
\mu_{k}:=
\begin{cases}
\frac{k!}{2^{k/2}(k/2)!} & \mbox{{\normalfont if $k$ is even,} }\\
\hfil 0 & \mbox{{\normalfont otherwise }} 
\end{cases}
\end{align*}
is the $k$-th standard Gaussian moment.

\begin{proof}
Let $[N]:=\{ 1,2,\ldots,N \}.$ Then similar to Lemma \ref{lem:mean} and \ref{lem:var}, for each $\boldsymbol{n} \in  [N]^k$, we have
\begin{align*}
&\frac{1}{X} \int_{-\infty}^{\infty} \prod_{j=1}^k 
\cos \left( \pi \breve{n_j} (2x+\delta)+\varphi \right)W \left( \frac{x}{X} \right) dx \\ 
=&\frac{1}{2^k} \sum_{\boldsymbol{\epsilon} \in \{\pm 1\}^k} 
 \prod_{j=1}^k e^{i \epsilon_j (\pi \breve{n_j} \delta + \varphi) } 
 \cdot \widehat{W} \left( -\left( \sum_{j=1}^k \epsilon_j 
\breve{n_j} \right)X \right)\\
=&\frac{1}{2^k} \sum_{\substack{ \boldsymbol{\epsilon} \in \{\pm 1\}^k \\ \langle \boldsymbol{\epsilon}, \boldsymbol{\breve{n}} \rangle=0}}
\prod_{j=1}^k  e^{i \epsilon_j \varphi } 
+O_A \left( \frac{1}{2^k} \sum_{\substack{ \boldsymbol{\epsilon} \in \{\pm 1\}^k \\ \langle \boldsymbol{\epsilon}, \boldsymbol{\breve{n}} \rangle \neq 0}} \left(\left| \langle \boldsymbol{\epsilon}, \boldsymbol{\breve{n}} \rangle  \right|X \right)^{-A} \right),
\end{align*}
where $ \langle \boldsymbol{\epsilon}, \boldsymbol{\breve{n}} \rangle:=\sum_{j=1}^k \epsilon_j \breve{n_j}.$ Since by Lemma \ref{lem:diss}
\begin{align*}
\left| \langle \boldsymbol{\epsilon}, \boldsymbol{\breve{n}} \rangle  \right| 
\geq \frac{m/D^{1/m}}{\left(kN^{1/m}\right)^{m^{k}-1}}
\end{align*}
whenever $\langle \boldsymbol{\epsilon}, \boldsymbol{\breve{n}} \rangle \neq 0,$ the error term here is
\begin{align*}
\ll_{A} \left( \frac{X}{\left( kN^{1/m} \right)^{m^{k}-1}} \right)^{-A} 
&\ll_{A,k} 
N^{A (m^{k}-1)/m} X^{-A} \\
&\ll_{A,k} X^{-\theta A}.
\end{align*}
Therefore, we have
\begin{align*}
 \left(\frac{2w}{\pi}\right)^k \underset{\langle \boldsymbol{\epsilon}, \boldsymbol{\breve{n}} \rangle \neq 0}{\sum_{\boldsymbol{\epsilon} \in \{\pm 1\}^k}\sum_{\boldsymbol{n}\in [N]^k}}
 \prod_{j=1}^k \frac{ \lambda_f(n_j)}{\sqrt{n_j}} \cdot 
 \frac{\sin \left( \pi \breve{n_j}\delta \right)}{\sqrt{\breve{n_j}}} \cdot 
 O_{A,k} \left(X^{-\theta A}  \right) 
\ll_{A,k}& \, \delta^k N^k X^{-\theta A}\\
\ll_{B,k}& \, \delta^k  X^{-B},
\end{align*}
so that
\begin{align*}
 \mathbb{E}_{x \sim X}^W( \Delta_f(x,\delta;N)^k )=& \left(\frac{w}{\pi}\right)^k \underset{\langle \boldsymbol{\epsilon}, \boldsymbol{\breve{n}} \rangle=0}{\sum_{\boldsymbol{\epsilon} \in \{\pm 1\}^k}\sum_{\boldsymbol{n}\in [N]^k}}
 \prod_{j=1}^k \frac{ \lambda_f(n_j)}{\sqrt{n_j}} \cdot 
 \frac{\sin \left( \pi \breve{n_j}\delta \right)}{\sqrt{\breve{n_j}}} \cdot e^{i\epsilon_j \varphi} 
+O_{B,k} (  \delta^k X^{-B} ).
\end{align*}

It remains to address the diagonal sum. Appealing to Lemma \ref{lem:li}, if $\langle \boldsymbol{\epsilon}, \boldsymbol{\breve{n}} \rangle=0,$ then there must exist a decomposition 
\begin{align*}
[k]:=\{1,2, \ldots,k \}= \bigsqcup_{i=1}^l S_i
\end{align*}
such that for $i=1,\ldots,l$ and $j \in S_i,$ we have $n_j=q_i r_j^m$ with $q_i$'s being distinct $m$-th power-free and with $r_j$'s satisfying
\begin{align*}
\langle \boldsymbol{\epsilon}, \boldsymbol{r}\rangle|_{S_i}:=\sum_{j \in S_i} \epsilon_j r_j=0.
\end{align*}
Therefore, the diagonal sum is
\begin{align*}
\left(\frac{w}{\pi}\right)^k \underset{\langle \boldsymbol{\epsilon}, \boldsymbol{\breve{n}} \rangle=0}{\sum_{\boldsymbol{\epsilon} \in \{\pm 1\}^k}\sum_{\boldsymbol{n}\in [N]^k}}
 \prod_{j=1}^k \frac{ \lambda_f(n_j)}{\sqrt{n_j}} \cdot 
 \frac{\sin \left( \pi \breve{n_j}\delta \right)}{\sqrt{\breve{n_j}}} \cdot e^{i\epsilon_j \varphi} \\
=\sum_{l=1}^k \sum_{[k]=\bigsqcup_{i=1}^l S_i}
\sum_{\substack{\boldsymbol{q} \in [N]^l\\ \text{$q_i$ distinct $m$-th power-free}}}
\prod_{i=1}^l  D_{q_i}(S_i),
\end{align*}
which is
\begin{gather*}
\sum_{\substack{[k]=\bigsqcup_{i=1}^{k/2} S_i\\|S_i|=2}}
\sum_{\substack{\boldsymbol{q} \in [N]^l\\ \text{$q_i$ distinct $m$-th power-free}}}
\prod_{i=1}^l  D_{q_i}(S_i)+
\sum_{l=1}^k \sum_{\substack{[k]=\bigsqcup_{i=1}^{l} S_i\\ \exists \, i \,:\, |S_i| \neq 2}}
\sum_{\substack{\boldsymbol{q} \in [N]^l\\ \text{$q_i$ distinct $m$-th power-free}}}
\prod_{i=1}^l  D_{q_i}(S_i) \\
=:\Sigma_1+\Sigma_2.
\end{gather*}

Let us deal with the second sum $\Sigma_2$ first. Since $D_q(S)=0$ if $|S|=1$, we can assume that $|S_i|\geq 2$ for all $i=1,\ldots,l.$ Therefore, Lemma \ref{lem:sumdqs} implies that 
\begin{align} \label{eq:2ndsumhighmoments}
\Sigma_2 \ll_{k} \sigma(\delta;N)^k \delta^{3m\eta.}
\end{align}

On the other hand, note that the first sum $\Sigma_1$ vanishes if $k$ is odd. Otherwise, if $k$ is even, then we have
\begin{align*}
\Sigma_1&=\sum_{\substack{[k]=\bigsqcup_{i=1}^{k/2} S_i\\|S_i|=2}}
\prod_{i=1}^{k/2} \sum_{\substack{q \leq N\\ \text{\normalfont $m$-th power-free}}}D_q(S_i)
-
\sum_{\substack{[k]=\bigsqcup_{i=1}^{k/2} S_i\\|S_i|=2}}
\sum_{\substack{\boldsymbol{q} \in [N]^{k/2}\\ \text{$q_i$ non-distinct $m$-th power-free}}}
\prod_{i=1}^{k/2}  D_{q_i}(S_i),
\end{align*}
where the first sum here is simply $\mu_k \cdot \sigma_f(\delta;N)^k$ by Lemma \ref{lem:sumdqs}.

On the other hand, since
\begin{align*}
D_q(S) \ll_{k} \frac{E_q(S)}{q^{\left(\frac{1}{3}+\frac{1}{2m}+\eta\right)|S|}},
\end{align*}
where $E_q(S)$ is defined in (\ref{eq:eqs}), it follows from the trivial case of (\ref{eq:eqsbound}) that $D_q(S) \ll_{k} 
\delta^{2m/3}.$ Therefore, the second sum is 
$\ll_{k} \sigma(\delta;N)^{k-2}\delta^{2m/3}$ again by Lemma \ref{lem:sumdqs}. Then, combining with (\ref{eq:2ndsumhighmoments}), the proposition follows from Lemma \ref{lem:varianceN}.
\end{proof}

\end{proposition}

\section{$L^2$-deviation of the approximants}

For degrees $m \geq 3,$ the pointwise error term provided in Proposition \ref{prop:fi} is insufficient to prove Theorem \ref{thm:real}. Inspired by the work of Lester \cite{MR3556248} and J\"{a}\"{a}saari \cite{MR4206430}, 
we establish bounds for the $L^2$-deviation of the approximants $\Delta_f(x,\delta;N)$ and $(x+\delta)^{-\frac{m-1}{2}}\Delta_f((x+\delta)^m),$ except that in order to prove Theorem \ref{thm:real}, our truncation is at some suitable $N=X^{o(1)}$ instead of a fixed power of $X$ as $X \to \infty.$

\begin{lemma} \label{prop:GLH}
Assume GLH. Let $X, \delta>0$ be real numbers satisfying the conditions in Theorem \ref{thm:real}. Then there exists $N \geq 2$ for which $\delta^m N \to \infty$ as $X\to \infty$ such that
\begin{align*} 
\mathbb{E}^W_{x \sim X} \left( \left( \Delta_f(x,\delta)-\Delta_f(x,\delta;N) \right)^2 \right) = o(\sigma_f(\delta)^2).
\end{align*}
For $m=2,$ the assumption of GLH can be dropped.
\end{lemma}

\begin{proof}
Let us denote
\begin{align*}
\Delta_f(x^m;N) :=  \frac{w}{\pi}\cdot x^{\frac{m-1}{2}}\sum_{n \leq N} 
\frac{\lambda_f(n)}{\sqrt{n}} \cdot \frac{\sin\left( 2\pi \breve{n}x+\varphi\right)}{\sqrt{\breve{n}}}.
\end{align*}
Then by definition, we have 
\begin{gather*}
\left( \Delta_f(x,\delta)-\Delta_f(x,\delta;N) \right)^2 \\
=\left( \left( \frac{\Delta_f((x+\delta)^m)}{(x+\delta)^{\frac{m-1}{2}}}
-\frac{\Delta_f(x^m)}{x^{\frac{m-1}{2}}}  \right)-\left( \frac{\Delta_f((x+\delta)^m;N)}{(x+\delta)^{\frac{m-1}{2}}}
-\frac{\Delta_f(x^m;N)}{x^{\frac{m-1}{2}}} \right) \right)^2 \\
\leq 2 \left( \frac{\Delta_f((x+\delta)^m)-\Delta_f((x+\delta)^m;N)}{(x+\delta)^{\frac{m-1}{2}}} \right)^2 
+  2 \left( \frac{\Delta_f(x^m)-\Delta_f(x^m;N)}{x^{\frac{m-1}{2}}} \right)^2.
\end{gather*}
Therefore, it suffices to bound
\begin{align} \label{eq:devN}
\frac{1}{X} \int_{-\infty}^{\infty} \left(  \frac{\Delta_f(x^m)-\Delta_f(x^m;N)}{x^{\frac{m-1}{2}}} \right)^2 W\left( \frac{x}{X} \right) dx.
\end{align}

For $m=2,$ Proposition \ref{prop:fi} 
yields
\begin{align*}
 \frac{\Delta_f(x^m)-\Delta_f(x^m;N)}{x^{\frac{m-1}{2}}} 
=\frac{w}{\pi}\sum_{N<n \leq x^m} \frac{\lambda_f(n)}{\sqrt{n}} \cdot \frac{\sin\left( 2\pi \breve{n}x+\varphi\right)}{\sqrt{\breve{n}}}
+O_{\epsilon}(x^{-\frac{1}{2}+\epsilon}).
\end{align*}
Arguing similarly as in Lemma \ref{lem:varianceN}, one can show that (\ref{eq:devN}) is
\begin{align*}
\ll_{\epsilon} \sum_{N<n \leq (5X/2)^m} \frac{\lambda_f(n)^2}{n}
+X^{-1+2\epsilon}.
\end{align*}
Assuming (\ref{eq:rankinselberg}), the sum here is
\begin{align*}
\ll \frac{\log^{r-1} N}{N^{1/m}} =o(\sigma_f(\delta)^2).
\end{align*}
Also, since $\log \frac{1}{\delta}=o(\log X)$ as $X \to \infty,$ we have $X^{-1+2\epsilon}=o(\sigma_f(\delta)^2)$ as well.

It remains to bound (\ref{eq:devN}) for $m \geq 3$ under GLH. Similar to \cite[Proposition 2.2]{MR3556248} and \cite[Proposition 9]{MR4206430}, by truncating and shifting the contour of the integral
\begin{align*}
\int_{1+\epsilon-i\infty}^{1+\epsilon+i\infty} L(f,s)\frac{x^{ms}}{s} ds
\end{align*}
appropriately, one can show that 
\begin{gather} 
\Delta_f(x^m)=\sum_{n \leq x^m} \lambda_f(n) - \mathop{\mathrm{Res}}_{s=1} L(f,s)\frac{x^{ms}}{s} \nonumber \\
= \Re\left(\frac{1}{\pi i}\left\{  \int_{-\epsilon}^{-\epsilon+iT_0}
+ \int_{\frac{1}{2}+iT_0}^{\frac{1}{2}+iT}
+ \int_{-\epsilon+iT_0}^{\frac{1}{2}+iT_0}
+ \int_{\frac{1}{2}+iT}^{1+\epsilon+iT}
+ \int_{1+\epsilon+iT}^{1+\epsilon+i\infty}
\right\}
L(f,s)x^{ms} \frac{ds}{s}
\right)+L(f,0) \label{eq:l(f,0)}
\end{gather}
for $1 \leq T_0 < T \leq x^m$ to be determined later.
Appealing to a saddle-point estimate due to Friedlander and Iwaniec (see \cite[pp. 497--499]{MR2134400} as well as \cite[Lemma 2.4]{MR3556248}), there exists a positive increasing function $\psi(x) \ll_{\epsilon} x^{\epsilon}$ for which $\psi(x) \to \infty$ as $x \to \infty$ such that 
\begin{gather} 
\Re\left(\frac{1}{\pi i}\int_{-\epsilon}^{-\epsilon+iT_0} L(f,s)x^{ms} \frac{ds}{s}
\right) \nonumber \\
=\Delta_f(x^m;N)+O\left( x^m \psi(x)^m {T_0}^{-\frac{m+1}{2}}
+\psi(x)^m {T_0}^{\frac{m}{2}-1}
\right), \label{eq:fisaddle}
\end{gather}
where we set $T_0= 2 \pi (N/D)^{\frac{1}{m}}x.$ Using the assumption (\ref{eq:rankinselberg}), one can show that
\begin{align} \label{eq:rsperron}
\int_{1+\epsilon+iT}^{1+\epsilon+i\infty} L(f,s)x^{ms} \frac{ds}{s} 
\ll_{\epsilon} & \, x^{m \epsilon }\sum_{n=1}^{\infty} \frac{|\lambda_f(n)|}{n^{1+\epsilon}\max\{ 1, T|\log(x^m/n) |\}} \nonumber \\
\ll_{\epsilon} & \, x^{(1+2\epsilon)m} T^{-1}
\end{align}
(see \cite[p. 71]{koukoulopoulos2020distribution} for instance). Also, assuming GLH, we have
\begin{align} \label{eq:glhhor1}
\int_{\frac{1}{2}+iT}^{1+\epsilon+iT} L(f,s)x^{ms} \frac{ds}{s}
\ll_{\epsilon} & \, (DT^m)^{\epsilon} \frac{1}{T} \int_{1/2}^{1+\epsilon} x^{m\sigma} d\sigma \nonumber\\
\ll_{\epsilon} & \, x^{(1+\epsilon)m} T^{-1+m\epsilon}.
\end{align}
Applying the Phragmén--Lindelöf principle, we also have
\begin{align} \label{eq:glhhor2}
\int_{-\epsilon+iT_0}^{\frac{1}{2}+iT_0} L(f,s)x^{ms} \frac{ds}{s}
\ll_{\epsilon} & \, \frac{1}{T_0} \int_{-\epsilon}^{1/2} (DT_0^m)^{\frac{1}{2}-\sigma+\epsilon} x^{\sigma} d\sigma \nonumber \\
\ll_{\epsilon} & \, x^{-m\epsilon} {T_0}^{\frac{m}{2}-1+2m\epsilon}.
\end{align}
Then, combining (\ref{eq:l(f,0)}), (\ref{eq:fisaddle}), (\ref{eq:rsperron}), (\ref{eq:glhhor1}) and (\ref{eq:glhhor2}) with the fact that $L(f,0) \ll D^{1/2}$ (see \cite[p. 497]{MR2134400} for instance), the difference $\Delta_f(x^m) - \Delta_f(x^m;N)$ is
\begin{gather*}
\ll_{\epsilon} 
\left| \int_{\frac{1}{2}+iT_0}^{\frac{1}{2}+iT} L(f,s) \frac{x^{ms}}{s} ds \right|
+D^{1/2}+x^m \psi(x)^m {T_0}^{-\frac{m+1}{2}} +\psi(x)^m {T_0}^{\frac{m}{2}-1}
\\
+ x^{(1+2\epsilon)m} T^{-1}
+  x^{(1+\epsilon)m} T^{-1+m\epsilon} + x^{-m\epsilon} {T_0}^{\frac{m}{2}-1+2m\epsilon},
\end{gather*}
so that by the Cauchy--Schwarz inequality, the expression (\ref{eq:devN}) is 
\begin{gather} 
\ll_{\epsilon} \frac{1}{X} \int_{-\infty}^{\infty} 
\left| \int_{\frac{1}{2}+iT_0}^{\frac{1}{2}+iT} L(f,s) x^{ms} \frac{ds}{s} \right|^2
W \left( \frac{x}{X} \right) \frac{dx}{x^{m-1}}
+\frac{X^{m+1+4m\epsilon}}{T^{2-2m\epsilon}}
\nonumber \\
+\frac{X^{m+1}\psi(X)^{2m}}{T_0^{m+1}}
+\frac{\psi(X)^{2m}T_0^{m-2+4m\epsilon}}{X^{m-1}}
\nonumber \\
\ll_{\epsilon} \frac{1}{X} \int_{-\infty}^{\infty} 
\left| \int_{\frac{1}{2}+iT_0}^{\frac{1}{2}+iT} L(f,s) x^{ms} \frac{ds}{s} \right|^2
W \left( \frac{x}{X} \right) \frac{dx}{x^{m-1}}
+\frac{X^{m+1+4m\epsilon}}{T^{2-2m\epsilon}} \nonumber \\
+\frac{\psi(X)^{2m}}{N^{1+\frac{1}{m}}}
+\frac{\psi(X)^{2m} N^{1-\frac{2}{m}+4\epsilon}}{X^{1-4m\epsilon}}. \label{eq:psi}
\end{gather}

It remains to bound the first integral, which is
\begin{gather} \label{eq:finale}
\iint_{(u,v) \in [T_0,T]^2} \frac{L(f,\frac{1}{2}+iu)L(f,\frac{1}{2}-iv)}{(\frac{1}{2}+iu)(\frac{1}{2}-iv)} \left( \frac{1}{X} \int_{-\infty}^{\infty} x^{1+im(u-v)} W\left( \frac{x}{X} \right) dx \right) du dv.
\end{gather}
Since $\widehat{W}$ decays rapidly, by making the change of variables $y=\frac{x}{X},$ the inner integral is
\begin{align*}
X^{1+im(u-v)} \int_{-\infty}^{\infty} y^{1+im(u-v)} W(y) dy
\leq g(A) X \min \{ 1, |u-v|^{-A} \}
\end{align*}
for some positive increasing function $g(A).$
Let $\theta(X)$ be a positive (slowly) decreasing function for which $\theta(X)=o_{X \to \infty}(1)$ and $g(\theta(X)^{-1}) \leq \log X$. Then the expression (\ref{eq:finale}) is
\begin{gather*} 
\leq  X  g(A) \iint_{\substack{(u,v) \in [T_0,T]^2\\|u-v| \leq X^{2\theta(X)}}} 
\left|  \frac{L(f,\frac{1}{2}+iu)L(f,\frac{1}{2}-iv)}{(\frac{1}{2}+iu)(\frac{1}{2}-iv)} \right|dudv  \nonumber \\
+ X^{1-2A\theta(X)}g(A) \iint_{\substack{(u,v) \in [T_0,T]^2\\|u-v| >X^{2\theta(X)} }} 
\left|  \frac{L(f,\frac{1}{2}+iu)L(f,\frac{1}{2}-iv)}{(\frac{1}{2}+iu)(\frac{1}{2}-iv)} \right| dudv. \label{eq:dichotomy}
\end{gather*}
Assuming GLH,  there exists a positive increasing function $\phi(x) \ll_{\epsilon} x^{\epsilon}$ satisfying $\phi(x) \to \infty$ as $x \to \infty$ for which the integrands here are $\leq (\phi(u)\phi(v))^m(uv)^{-1}.$ Then this is
\begin{align*}
\ll & X^{1+2\theta(X)}g(A) \phi(2T_0)^{2m}T_0^{-1}+X^{1-2A\theta(X)}g(A)
\phi(T)^{2m}\log^2 T \\
= & X^{2\theta(X)} g(A) \phi(4\pi (N/D)^{1/m}X)^{2m}N^{-\frac{1}{m}}
+ X^{1-2A\theta(X)}g(A)
\phi(T)^{2m}\log^2 T.
\end{align*}
Finally, combining with (\ref{eq:psi}), the lemma follows from taking $T=(X/2)^m, A=\theta(X)^{-1}, N=\delta^{-m}
g(\theta(X)^{-1} )^m X^{2m\theta(X)} \psi(X)^{2m}\phi(X^2)^{2m^2}$ and letting $\epsilon \to 0^+.$
\end{proof}

\begin{lemma} \label{lem:grh}
Let $X, \delta>0$ be real numbers satisfying the conditions in Theorem \ref{thm:real}. Then as $X \to \infty,$ we have
\begin{align*}
\mathbb{E}^W_{x \sim X} \left( 
\left(
\frac{\Delta_f((x+\delta)^m)}{(x+\delta)^{\frac{m-1}{2}}}-
\frac{\Delta_f((x+\delta)^m)}{x^{\frac{m-1}{2}}}
\right)^2
\right)=o(\sigma_f(\delta)^2).
\end{align*}

\end{lemma}

\begin{proof}
Since
\begin{align*}
\frac{\Delta_f((x+\delta)^m)}{(x+\delta)^{\frac{m-1}{2}}}
=\frac{\Delta_f((x+\delta)^m)}{x^{\frac{m-1}{2}}}
+O \left( \delta \cdot \frac{\Delta_f((x+\delta)^m)}{x^{\frac{m+1}{2}}} \right),
\end{align*}
it suffices to bound
\begin{align*}
\frac{1}{X} \int_{-\infty}^{\infty} \left( \frac{\Delta_f(x^m)}{x^{\frac{m+1}{2}}} \right)^2 W \left( \frac{x}{X} \right) dx,
\end{align*}
which is
\begin{align*}
\ll \frac{1}{X} \int_{-\infty}^{\infty} \left( \frac{\Delta_f(x^m)-\Delta_f(x^m;N)}{x^{\frac{m+1}{2}}} \right)^2 W \left( \frac{x}{X} \right) dx
+ \frac{1}{X} \int_{-\infty}^{\infty} \left( \frac{\Delta_f(x^m;N)}{x^{\frac{m+1}{2}}} \right)^2 W \left( \frac{x}{X} \right) dx
\end{align*}
with $N=\delta^{-m}
g(\theta(X)^{-1} )^m X^{2m\theta(X)} \psi(X)^{2m}\phi(X^2)^{2m^2}$ as defined in the last line of the proof of Lemma \ref{prop:GLH}. Appealing to Lemma \ref{prop:GLH} and (\ref{eq:devN}), the first integral is $o(\sigma_f(\delta)^2 X^{-2}).$ 

On the other hand, arguing similarly as in Lemma \ref{lem:varianceN}, the second integral is
\begin{align*}
\ll \frac{1}{X^2} \sum_{n \leq N} \frac{\lambda_f(n)^2}{n^{1+\frac{1}{m}}} 
\ll\frac{1}{X^2} \cdot \frac{\log^{r-1} N}{N^{1/m}},
\end{align*}
which is again $o(\sigma_f(\delta)^2 X^{-2}).$ Therefore, the proof is completed.
\end{proof}

\section{Proof of Theorem \ref{thm:real}}

We are now set to prove Theorem \ref{thm:real}. Given a Borel subset $B \subseteq \mathbb{R},$ we denote
\begin{align*}
\mathbb{P}_{x \sim X}^W \left( B \right):=\frac{1}{X}\int_B W\left( \frac{x}{X} \right) dx.
\end{align*}

Let $N=\delta^{-m}
g(\theta(X)^{-1} )^m X^{2m\theta(X)} \psi(X)^{2m}\phi(X^2)^{2m^2}$ as defined in the last line of the proof of Lemma \ref{prop:GLH}. In particular, it satisfies  $\delta^{m} N \to \infty$ but  $\log N =o \left( \log X\right)$ as $X \to \infty,$ so that Lemma \ref{lem:mean}, Lemma \ref{lem:var} and Proposition \ref{prop:hm} are all applicable. Then, using the method of moments (see \cite[Theorem 30.1]{MR1324786} for instance), 
we have
\begin{align*}
\mathbb{P}_{x \sim X}^W \left( \frac{\Delta_f(x,\delta;N)}{\sigma_f(\delta;N)} \in (\alpha, \beta] \right)
=\frac{1}{\sqrt{2\pi}}\int_{\alpha}^{\beta}e^{-\frac{t^2}{2}}dt+o_{\alpha,\beta}(1)
\end{align*}
for any real numbers $\alpha<\beta.$ Since
$\sigma_f(\delta;N)=(1+o(1))\sigma_f(\delta)$
by Lemma \ref{lem:varianceN}, it follows that
\begin{align*}
\mathbb{P}_{x \sim X}^W \left( \frac{\Delta_f(x,\delta;N)}{\sigma_f(\delta)} \in (\alpha, \beta] \right)
=\frac{1}{\sqrt{2\pi}}\int_{\alpha}^{\beta}e^{-\frac{t^2}{2}}dt+o_{\alpha,\beta}(1).
\end{align*}
Appealing to Lemma \ref{prop:GLH}, we have
\begin{align*}
\mathbb{E}^W_{x \sim X} \left( \left( \Delta_f(x,\delta)-\Delta_f(x,\delta;N) \right)^2 \right)=o(\sigma_f(\delta)^2),
\end{align*}
so that Chebyshev's inequality gives
\begin{align*}
\mathbb{P}_{x \sim X}^W \left( \frac{\Delta_f(x,\delta)}{\sigma_f(\delta)} \in (\alpha, \beta] \right) 
=\frac{1}{\sqrt{2\pi}}\int_{\alpha}^{\beta}e^{-\frac{t^2}{2}}dt+o_{\alpha,\beta}(1).
\end{align*}
Similarly, since
\begin{align*}
\mathbb{E}^W_{x \sim X} \left( 
\left(
\frac{\Delta_f((x+\delta)^m)}{(x+\delta)^{\frac{m-1}{2}}}-
\frac{\Delta_f((x+\delta)^m)}{x^{\frac{m-1}{2}}}
\right)^2
\right)=o(\sigma_f(\delta)^2)
\end{align*}
by Lemma \ref{lem:grh}, we arrive at
\begin{align*}
\mathbb{P}_{x \sim X}^W \left( 
\frac{  \Delta_f((x+\delta)^m)-\Delta_f(x^m)  }{x^{\frac{m-1}{2}}\sigma_f(\delta)}
\in (\alpha, \beta] \right)=\frac{1}{\sqrt{2\pi}}\int_{\alpha}^{\beta}e^{-\frac{t^2}{2}}dt+o_{\alpha,\beta}(1).
\end{align*}

Finally, let $0<\epsilon<1.$ We choose $W$ for which $W \equiv 1$ on $(1+\epsilon, 2-\epsilon).$ Since $W$ is nonnegative and of unit mass, we have
\begin{gather*}
\frac{1}{X}\mathop{\mathrm{meas}}\left\{ x \in [X,2X] \,: \, \frac{  \Delta_f((x+\delta)^m)-\Delta_f(x^m) }{x^{\frac{m-1}{2}}\sigma_f(\delta)} \in (\alpha,\beta] \right\} \\
- \mathbb{P}_{x \sim X}^W \left( 
\frac{  \Delta_f((x+\delta)^m)-\Delta_f(x^m)  }{x^{\frac{m-1}{2}}\sigma_f(\delta)}
\in (\alpha, \beta] \right) \\
=\frac{1}{X}\int_{-\infty}^{\infty} 1_{(\alpha,\beta]} 
\left( \frac{  \Delta_f((x+\delta)^m)-\Delta_f(x^m)  }{x^{\frac{m-1}{2}}\sigma_f(\delta)} \right)
\left( 1_{[1,2]}\left( \frac{x}{X} \right)-W\left( \frac{x}{X} \right) \right) dx \\
\leq  \int_{\mathbb{R} \setminus (1+\epsilon, 2-\epsilon)} 
\left(1_{[1,2]}(x)+W(x) \right) dx = 4\epsilon.
\end{gather*}
Letting $\epsilon \to 0^+,$ the proof is completed.


\section{Uniformity in $L$-functions}

Since it is unclear how the Rankin--Selberg type partial sum estimate (\ref{eq:rankinselberg}) depends on $L(f,s)$ in general, our main result, Theorem \ref{thm:real}, lacks uniformity in the parameters of $L$-functions. Nevertheless, in certain special cases, such as Example \ref{eq:tauk} where $\lambda_f(n)=\tau_k(n),$ we claim that by carefully tracking all dependencies on $k,$ if $\delta=o\left(\frac{1}{k} \right)$ but $\log \left(\frac{1}{k\delta} \right)=o\left( \frac{1}{k}\log X\right)$ as $X \to \infty,$ then the theorem holds uniformly in $k=k(X).$ 

\section*{Acknowledgements}
The author is grateful to Andrew Granville for his advice and encouragement. He would also like to thank Bingrong Huang, Steve Lester and Maksym Radziwiłł for many helpful discussions, as well as Cihan Sabuncu and the anonymous referee for their comments and corrections. Part of the work was supported by the Swedish Research Council under grant no. 2016-06596 while the author was in residence at Institut Mittag-Leffler in Djursholm, Sweden during the semester of Winter 2024.

\printbibliography

\end{document}